\documentclass[12pt,reqno]{amsart}

\usepackage[T1]{fontenc}
\usepackage[utf8]{inputenc}
\usepackage{amsmath,amsfonts,amssymb,amsthm,mathtools,mathrsfs, booktabs}
\usepackage{tikz}
\usepackage{hyperref}
\hypersetup{hidelinks}
\usepackage{microtype}

\numberwithin{equation}{section}

\theoremstyle{plain}
\newtheorem{theorem}{Theorem}[section]
\newtheorem{lemma}[theorem]{Lemma}
\newtheorem{proposition}[theorem]{Proposition}

\theoremstyle{definition}
\newtheorem{definition}[theorem]{Definition}
\newtheorem{question}[theorem]{Question}
\newtheorem{example}[theorem]{Example}
\newtheorem{remark}[theorem]{Remark}

\newcommand{\Dhalf}{\mathscr{D}^{1/2}}

\title[Discovering The Half-Derivative from First Principles]{Experimental Mathematics: Discovering The Half-Derivative from First Principles}

\author{H\`ung Vi\d{\^e}t Chu}
\address{Department of Mathematics, Washington and Lee University, Lexington, VA 24450, USA}
\email{hchu@wlu.edu}

\author{Steven J. Miller}
\address{Department of Mathematics and Statistics, Williams College, Williamstown, MA 01267, USA}
\email{sjm1@williams.edu}

\author{Joshua M. Siktar}
\address{Department of Mathematics, Texas A\&M University, College Station, TX 77843, USA}
\email{jmsiktar@tamu.edu}

\subjclass[2020]{26A33 (primary); 97I40 (secondary)}
\keywords{fractional derivative, Gamma function, functional equation}

\begin{document}

\begin{abstract}
This expository paper grew out of two lectures in Miller's Math 409: The little Questions (Williams College, Fall 2026). We explore fractional derivatives by asking what a ``half-derivative'' might mean if applying it twice is to recover an ordinary derivative. Linearity and a natural guess for the half-derivative of a power function lead to a functional equation related to the classical factorial recurrence. Choosing the familiar Gamma-function extension gives the classical formula
\[
  \mathscr{D}^{1/2}(x^n)
  \ =\ \frac{\Gamma(n+1)}{\Gamma(n+\tfrac12)}x^{n-1/2}.
\]
We derive $\Gamma(n+1)=n!$, evaluate $\Gamma(1/2)=\sqrt{\pi}$ using the Gaussian integral, and verify the half-derivative formula with an elementary integral. 
Along the way, we explore how the Gamma function generalizes factorials and binomial coefficients beyond their classical domains, encounter the double factorial, see why the usual product rule fails, and examine the half-derivative of $e^x$. We also explain why the original requirements do not determine a unique operator. The treatment emphasizes motivation, the benefits of experimentation through accessible calculations in test cases, rather than the general theory of fractional calculus.
\end{abstract}

\maketitle
\tableofcontents

\section{Introduction}
\label{sec:intro}

In first-year calculus, we learn two familiar notations for the derivative of a function $f$:
\[
  f'(x) \qquad \text{and} \qquad \frac{df}{dx}.
\]
Both describe the instantaneous rate of change, but each has its own advantages. The prime notation naturally extends to higher derivatives, written as $f^{(k)}$, while Leibniz's notation, $\frac{d^k f}{dx^k}$, explicitly indicates both the variable of differentiation and the order of the derivative.

The latter notation raises a natural, if initially whimsical, question: what would it mean to take a \emph{half-derivative}? More generally, can we make sense of
\[
  \frac{d^\alpha f}{dx^\alpha}
\]
when $\alpha$ is not an integer? This question leads to the subject of \emph{fractional calculus}, in which the familiar operations of differentiation and integration are extended to non-integer orders.

This idea is far from new. Its history dates back to Leibniz's correspondence with l'H\^opital in 1695, in which the possibility of differentiation of non-integer order was first raised \cite{Leibniz1695,Ross1975}. Subsequent contributions by Abel, Euler, Liouville, Riemann, and others led to the development of several approaches to fractional differentiation \cite{Miller, Ross1975}. In particular, Liouville's early contributions are presented in \cite{Liouville1832,Liouville1832_02}.

Modern treatments include extensions of repeated integration, leading to the Riemann-Liouville fractional operators, as well as other formulations such as the Caputo derivative and Fourier-based approaches. For systematic treatments, we refer to Oldham and Spanier \cite{OldhamSpanier} for an introduction to fractional calculus and its applications; Samko, Kilbas, and Marichev \cite{Samko} for a comprehensive mathematical treatment of fractional integrals and derivatives; Miller and Ross \cite{Miller} for their historical survey and development of fractional calculus; and Podlubny \cite{Podlubny} for fractional differential equations and their applications.

The purpose of this paper is primarily expository. Rather than beginning with an established definition of fractional differentiation, we take a more exploratory approach, guided by a simple question.
\begin{question}
What properties must an operator $\Dhalf$ satisfy if it is to deserve the name \emph{half-derivative}?
\end{question}
Our goal is not merely to introduce fractional derivatives and the Gamma function, but also to illustrate the process of mathematical discovery. In particular, we hope to encourage students to ask fruitful questions, identify promising starting points, experiment with simple examples, and use their observations as guidance toward deeper mathematical ideas. As emphasized by Miller and Wu \cite{MillerWu}, understanding how to discover a mathematical result can be just as valuable as understanding its proof. We hope that our approach will help students develop the curiosity and intuition needed to explore unfamiliar mathematical problems independently.

Our starting point is the need that applying $\Dhalf$ twice should recover the ordinary derivative, at least for simple functions such as powers of $x$. Together with linearity and a suitable guess for the half-derivative of a power function, this leads to a functional equation for the coefficients. By expressing these coefficients as ratios of functions, we arrive at the familiar factorial recurrence, but with an interesting twist: it must hold not only for integers, but also for half-integers! Choosing the classical Gamma-function extension then suggests a formula for the half-derivative of a power function. The original requirements do not force this choice; that distinction will be part of our story.

Along the way, we explore several related ideas. We derive the fundamental properties of the Gamma function, including its connection to the Gaussian integral, and examine the half-derivatives of constants and exponential functions. We also encounter double factorials and their combinatorial interpretation, their connection to the moments of the standard normal distribution, Taylor series, the failure of the ordinary product rule for fractional derivatives, and the Gaussian integral
\[
  \int_{-\infty}^{\infty} e^{-x^2/2}\,dx\ =\ \sqrt{2\pi}.
\] 
To conclude we briefly discuss derivatives of arbitrary real orders and a generalized (fractional) Leibniz rule. These results are classical; our emphasis is on how they arise naturally from elementary mathematical considerations and how they are connected.

This exposition grew out of two lectures in Miller's \emph{Math 409: The little Questions} at Williams College in Fall 2026. It is intended for undergraduate students who have completed single-variable calculus and have some familiarity with real analysis,
particularly infinite series and improper integrals. No prior knowledge of the Gamma function, fractional calculus, or complex
analysis is assumed. It is also a case study of applying simple techniques in innovative ways, including the method of undetermined coefficients and telescoping products. We hope that this approach illustrates how asking seemingly simple questions about familiar mathematical concepts can lead to rich and unexpected connections. 

\section{Desiderata for a fractional derivative}
\label{sec:desiderata}

What properties should the half-derivative have?  We propose two.

\begin{definition}[Composition property]\label{def:composition}
We would like a half-derivative to satisfy
\[
  \Dhalf\!\left(\Dhalf f\right)\ =\ \frac{df}{dx}
\]
for functions where the operations on both sides of the equation are individually well-defined. In particular, we hope this happens for ordinary polynomials and other familiar smooth functions.
\end{definition}

This is a guideline, not a promise that every possible function will satisfy it. The function on which we first act must be eligible for a second half-derivative, too. We'll see a simple limitation of this rule later. Similarly, we would hope that three applications of a ``one-third derivative'' give one ordinary derivative on appropriate functions.

\begin{definition}[Linearity]\label{def:linearity}
The operator $\Dhalf$ should be linear: that is,
\begin{equation}\label{Eq: linearity}
  \Dhalf(af + bg) \ = \  a\,\Dhalf f + b\,\Dhalf g
\end{equation}
for all constants $a, b$ and suitable functions $f, g$.
\end{definition}

Linearity packages two familiar properties: (i) the derivative of a constant multiple is that multiple times the derivative, and (ii) the derivative of a sum is the sum of the derivatives.  It is also the condition that makes $\Dhalf$ an operator on a function space, analogous to a linear map on a vector space, of which the standard derivative is an example. 

\begin{remark}[Non-uniqueness]
These two properties alone cannot pin down $\Dhalf$ uniquely. If $\Dhalf$ satisfies the composition property, so does $-\Dhalf$, since $(-\Dhalf)(-\Dhalf)f=(-1)^2\Dhalf(\Dhalf f)=f'$ wherever these expressions make sense.  This suggests looking for a particularly natural choice, one that also fits with the familiar rules of calculus.
\end{remark}

\section{The ansatz: powers of \texorpdfstring{$x$}{x}}
\label{sec:ansatz}

To find $\Dhalf$ concretely, we start with the simplest non-trivial test functions: monomials $x^n$ for non-negative integers $n$.  Furthermore, monomials are natural building blocks. The function $e^x$ is also an appealing point of entry because $(e^x)' = e^x$, but its Taylor series
\[
  e^x \ =\ \sum_{n=0}^\infty \frac{x^n}{n!}
\]
involves infinitely many monomials.  Understanding $\Dhalf(x^n)$ for each $n$ and then (assuming $\Dhalf$ commutes with the infinite sum) combining gives $\Dhalf(e^x)$. 

\subsection{Shape of the answer}

We want to compute $\Dhalf(x^n)$ when $n$ is a non-negative integer, knowing that
\[
  \Dhalf\!\bigl(\Dhalf(x^n)\bigr) \ =\  \frac{d}{dx}(x^n) \ =\  n\,x^{n-1}.
\]
Let's guess the functional form of $\Dhalf(x^n)$, with the following clues.

\begin{itemize}
  \item It seems reasonable to look for another power of $x$, since ordinary differentiation takes powers to powers.
  \item As in ordinary differentiation, we expect a coefficient depending on the exponent.
  \item To keep the search simple, we first try a single power, with no additional terms.
\end{itemize}

This motivates the following ansatz (educated guess):

\begin{equation}\label{eq:ansatz}
  \Dhalf(x^n) \ = \ A(n)\,x^{B(n)},
\end{equation}
where $A(n)$ and $B(n)$ are functions to be chosen. For this exploration, we use the same ansatz for the non-integer powers that arise after the first application of $\Dhalf$.

\subsection{Determining \texorpdfstring{$B(n)$}{B(n)}}

Applying $\Dhalf$ to both sides of~\eqref{eq:ansatz}, we obtain
\begin{align*}
  \Dhalf\!\bigl(\Dhalf(x^n)\bigr)
  &\ =\ \Dhalf\!\bigl(A(n)\,x^{B(n)}\bigr)\\
  &\ =\ A(n)\,\Dhalf\!\bigl(x^{B(n)}\bigr)\\
  &\ =\ A(n)\,A(B(n))\,x^{B(B(n))},
\end{align*}
where the second equality follows from \eqref{Eq: linearity}: the
half-derivative of a constant multiple of a function is
that same constant times the half-derivative of the function.

For positive integers $n$, our requirement that applying
$\Dhalf$ twice recovers the ordinary derivative means that
the expression above must equal $nx^{n-1}$.
Comparing coefficients and exponents, we arrive at two functional equations:
\begin{align}
  B(B(n)) &\ =\ n-1, \label{eq:Bfunc}\\
  A(n)\,A(B(n)) &\ =\ n. \label{eq:Afunc}
\end{align}

The equation \eqref{eq:Bfunc} can have other solutions\footnote{For example, one could define
$B(n)=n-\frac13$ for odd positive integers $n$ and
$B(n)=n-\frac23$ for even positive integers $n$,
and then prescribe
$B(n-\frac13)=n-1$ in the former case and
$B(n-\frac23)=n-1$ in the latter.
This gives $B(B(n))=n-1$ for every positive
integer $n$, illustrating that our choice is
not unique.}, but one particularly natural choice is 
\begin{equation}\label{e10}
  B(n) \ =\ n - \tfrac{1}{2},
\end{equation}
since 
$$B(B(n)) \ =\  B(n - \tfrac{1}{2}) \ =\  n - \tfrac{1}{2} - \tfrac{1}{2} \ =\ n - 1.$$ 
We adopt this choice, giving
\[
  \Dhalf(x^n)\ =\ A(n)\,x^{n-1/2}.
\]
We shall treat $n=0$ separately (see Subsection \ref{sec:negfact}): the ordinary derivative of $1$ is zero, so comparing nonzero powers does not justify the same exponent argument in that case.

\subsection{Determining \texorpdfstring{$A(n)$}{A(n)}: the functional equation}

With $B(n) = n - \frac{1}{2}$, Equation~\eqref{eq:Afunc} becomes
\begin{equation}\label{eq:Arecurrence}
  A(n)\,A\!\left(n - \tfrac{1}{2}\right) \ =\  n.
\end{equation}
We seek values of $A$ at integers and half-integers satisfying this relation. This does not, by itself, determine them uniquely.

Since direct inspection of~\eqref{eq:Arecurrence} is difficult, we write $A$ as a ratio
\[
  A(n) \ =\  \frac{f(n)}{g(n)}
\]
for functions $f$ and $g$ to be chosen, hoping that
substitution into \eqref{eq:Arecurrence} will lead to cancellations
between numerators and denominators. Indeed,
\[
  A(n)\,A\!\left(n - \tfrac{1}{2}\right) \ =\  \frac{f(n)}{g(n)} \cdot \frac{f(n - \tfrac{1}{2})}{g(n - \tfrac{1}{2})}\  =\ n.
\]
The key observation is that we can arrange telescoping cancellation by setting
\[
  g(n)\ =\ f\!\left(n - \tfrac{1}{2}\right),
\]
which gives
\[
  A(n)\,A\!\left(n - \tfrac{1}{2}\right)
  \ =\ \frac{f(n)}{f(n - \tfrac{1}{2})} \cdot \frac{f(n - \tfrac{1}{2})}{f(n - 1)}
  \ =\ \frac{f(n)}{f(n-1)}\ =\ n.
\]

With this additional choice, our problem reduces to finding a function $f$ satisfying the simple recurrence
\begin{equation}
  f(n) \ =\ n \cdot f(n-1),
\end{equation}
which is the defining recurrence of the factorial function!  Since multiplying $f$ by a nonzero constant does not
change the ratio $f(n)/f(n-1)$, we may normalize
$f(0)=1$, provided that $f(0)\neq 0$.\footnote{
If $f(0)=0$, the recurrence forces $f(n)=0$
for every nonnegative integer $n$, which is incompatible with our construction.}
We then obtain
\[
f(n)\ =\ n! \qquad \text{for } n=0,1,2,\ldots.
\]
Assembling the pieces gives
\begin{equation}\label{e11}
  A(n) \ =\  \frac{f(n)}{g(n)} \ =\ \frac{f(n)}{f(n - \tfrac{1}{2})} \ =\ \frac{n!}{(n - \tfrac{1}{2})!}.
\end{equation}
Now we face the obstacle of half-integer factorials: what does $(n - \tfrac{1}{2})!$ mean?  The factorial function is initially defined for non-negative integers.  For example, when $n = 0$, we need $\left(-\frac{1}{2}\right)!$, and when $n = 1$, we need $\frac{1}{2}!$. This suggests extending the factorial function to non-integer arguments. 
We will choose the classical \emph{Gamma extension}, but neither the factorial recurrence nor our starting requirements uniquely determine this choice.

\begin{remark}[How much freedom remains?]
There is more freedom than simply replacing a half-derivative by its negative. For example, once we find one set of coefficients $A(n)$, we could multiply its values at integer exponents by a nonzero constant $c$ and divide its values at half-integer exponents by $c$. The products $A(n)A(n-\tfrac12)$ would remain unchanged. The Gamma function gives one natural and consistent choice, not the only algebraic possibility. 
\end{remark}

\section{The Gamma function}
\label{sec:gamma}

\subsection{Definition and convergence}

Now we recall the classical Gamma function (see \cite[Eq.~5.2.1]{DLMF}).

\begin{definition}[Gamma function]\label{def:gamma}
For a real number $s>0$, define
\begin{equation}\label{Eq: Ga}
  \Gamma(s) \;:=\; \int_0^\infty e^{-x} x^{s-1}\,dx.
\end{equation}
\end{definition}

\begin{remark}
The integrand can also be written $e^{-x} x^s\,\frac{dx}{x}$.  The measure $\frac{dx}{x}$ is invariant under the rescaling $x \mapsto cx$ (for $c > 0$), a symmetry that makes the Gamma function well-adapted to multiplicative problems.
\end{remark}

\begin{proposition}[Convergence of $\Gamma(s)$]\label{prop:convergence}
For $s>0$, the integral defining $\Gamma(s)$ converges.
\end{proposition}

\begin{proof}
We check the two possible locations for singularities: near $x = 0$ and near $x = \infty$.

\medskip
\noindent\textit{Near $x = 0$:} Since $e^{-x} \to 1$ as $x \to 0^+$, the integrand behaves like $x^{s-1}$ near the origin. More precisely, we have that $\lim_{x \to 0}\frac{e^{-x}x^{s - 1}}{x^{s - 1}} = 1$.  We have
\[
  \int_0^1 x^{s-1}\,dx \ = \ \frac{x^s}{s}\Big|_0^1 \ =\  \frac{1}{s}
\]
provided $s>0$ (so that $x^s \to 0$ as $x \to 0^+$).

\medskip
\noindent\textit{Near $x = \infty$:} For any fixed $s>0$ and large $x$, the factor $e^{-x}$ decays much faster than any power of $x$ grows. More precisely, choose an integer $N>s$; then
\[
  x^{s-1} e^{-x} \ \leq\ \frac{x^{s-1}}{x^N/N!}\ =\ N!\,x^{s-1-N},
\]
since $e^x \geq x^N/N!$.  Since $s-1-N<-1$, the integral $\int_1^\infty x^{s-1-N}\,dx$ converges.
\end{proof}


\subsection{Values at positive integers}

Now we show that the Gamma function \eqref{Eq: Ga} admits simple recursive identities.

\begin{lemma}[Functional equation]\label{lem:functional}
For real $s>0$,
\begin{equation}\label{Eq: rec}
  \Gamma(s+1) \ =\ s\,\Gamma(s).
\end{equation}
\end{lemma}

\begin{proof}
Integrate by parts with $u = x^s$ and $dv = e^{-x}\,dx$ (so $du = sx^{s-1}\,dx$ and $v = -e^{-x}$):
\begin{align*}
  \Gamma(s+1) &\ =\ \int_0^\infty e^{-x} x^s\,dx \\
  &\ =\ \Bigl[-e^{-x} x^s\Bigr]_0^\infty + s\int_0^\infty e^{-x} x^{s-1}\,dx.
\end{align*}
The boundary term vanishes: at $x = \infty$, $e^{-x} x^s \to 0$ (exponential beats polynomial\footnote{This limit can be verified using L'Hopital's Rule.}); at $x = 0$, $x^s \to 0$ for $s>0$.  The remaining integral is exactly $\Gamma(s)$.
\end{proof}

\begin{theorem}[Gamma generalizes factorial]\label{thm:factorial}
For every positive integer $n$,
\[
  \Gamma(n)\ =\ (n-1)!.
\]
\end{theorem}

\begin{proof}
We have $\Gamma(1) = \int_0^\infty e^{-x}\,dx = 1$, so applying Lemma~\ref{lem:functional} repeatedly gives
\[
  \Gamma(n) \ =\  (n-1)\,\Gamma(n-1)\  =\ (n-1)(n-2)\,\Gamma(n-2)\ =\ \cdots\ =\ (n-1)!\,\Gamma(1)\ =\ (n-1)!.\qedhere
\]
\end{proof}

\subsection{Value at positive half-integers}
Thanks to Lemma \ref{lem:functional}, the Gamma function at half-integers is determined by $\Gamma(1/2)$. We provide a proof of its value next.

\begin{theorem}\label{thm:gammaHalf}We have
$\displaystyle\Gamma\!\left(\tfrac{1}{2}\right) = \sqrt{\pi}$.
\end{theorem}

\begin{proof}
We use the standard Gaussian integral
(see, for example, \cite{MillerProbability}):
\begin{equation}
\int_{-\infty}^{\infty} e^{-t^2/2}\,dt
\ =\ \sqrt{2\pi}.
\end{equation}
(We take this as known; its proof via Fubini and polar coordinates is a standard exercise in multivariable calculus.)

Since the integrand is even, the integral from $0$ to $\infty$ equals half the full integral:
\[
  \int_0^\infty e^{-t^2/2}\,dt \ =\ \frac{\sqrt{2\pi}}{2}.
\]
Make the substitution $x = t^2/2$, so $t = \sqrt{2x}$ and $dt = \frac{1}{\sqrt{2x}}\,dx = \frac{dx}{\sqrt{2x}}$, to obtain
\begin{align*}
  \int_0^\infty e^{-t^2/2}\,dt
  &\ =\ \int_0^\infty e^{-x} \cdot \frac{dx}{\sqrt{2x}} \\
  &\ =\ \frac{1}{\sqrt{2}} \int_0^\infty e^{-x} x^{-1/2}\,dx \\
  &\ =\  \frac{1}{\sqrt{2}}\,\Gamma\!\left(\tfrac{1}{2}\right).
\end{align*}
Setting the two expressions equal gives
\[
  \frac{1}{\sqrt{2}}\,\Gamma\!\left(\tfrac{1}{2}\right)\ =\ \frac{\sqrt{2\pi}}{2};
\]
hence,
\[
  \Gamma\!\left(\tfrac{1}{2}\right)\ =\ \sqrt{\pi}. \qedhere
\]
\end{proof}

\begin{remark}
If we extend the factorial notation by $t!=\Gamma(t+1)$, then $(-\tfrac12)!=\sqrt{\pi}$. This is not the number of arrangements of ``negative one-half objects'': the combinatorial meaning of factorial does not extend in that literal way. It is instead a useful analytic extension of familiar factorial identities.
\end{remark}

Using the functional equation $\Gamma(s+1) = s\,\Gamma(s)$ repeatedly, we get
\begin{align*}
  \Gamma\!\left(\tfrac{3}{2}\right) &\ =\ \tfrac{1}{2}\,\Gamma\!\left(\tfrac{1}{2}\right)\ =\  \frac{\sqrt{\pi}}{2}\ =\ \frac{1\cdot \sqrt{\pi}}{2},\\
  \Gamma\!\left(\tfrac{5}{2}\right) &\ =\ \tfrac{3}{2}\,\Gamma\!\left(\tfrac{3}{2}\right) \ =\ \frac{3\sqrt{\pi}}{4}\ =\ \frac{3\cdot 1\cdot \sqrt{\pi}}{4},\\
  \Gamma\!\left(\tfrac{7}{2}\right) &\ =\ \tfrac{5}{2}\,\Gamma\!\left(\tfrac{5}{2}\right)\ =\ \frac{15\sqrt{\pi}}{8}\ =\ \frac{5\cdot 3\cdot 1\cdot \sqrt{\pi}}{8},\\
  \Gamma\!\left(\tfrac{2n+1}{2}\right) &\ =\  \frac{(2n-1)!!\,\sqrt{\pi}}{2^n},
\end{align*}
where $(2n-1)!! = 1 \cdot 3 \cdot 5 \cdots (2n-1)$ is the double factorial (product of odd numbers up to $2n-1$), with $(-1)!!=1$ when $n=0$. The general formula is immediate from induction. 

\subsection{Extension to negative non-integers}

The functional equation $\Gamma(s)=\Gamma(s+1)/s$ also gives values at negative nonintegers. For example,
\[
  \Gamma\!\left(-\tfrac12\right)
  \ =\ \frac{\Gamma(\tfrac12)}{-\tfrac12}\ =\ -2\sqrt{\pi}.
\]
At $0,-1,-2,\ldots$, however, the Gamma function has no finite value: its magnitude tends to infinity as one approaches any of these numbers. They are called \emph{poles}.

A useful convention is to extend the reciprocal $1/\Gamma(s)$ by setting
\[
  \frac1{\Gamma(0)}\ =\ \frac1{\Gamma(-1)}\ =\ \frac1{\Gamma(-2)}\ =\ \cdots\ =\ 0.
\]
This extension is continuous at each of these points. We use this convention whenever a Gamma quotient has such a denominator.

\subsection{What should a negative factorial mean?}

We have seen that the factorial function can be extended beyond the nonnegative integers using the Gamma function. But this raises a natural question: what should $(-1)!$, $(-2)!$, or more generally $(-n)!$ mean when $n$ is a positive integer?

Rather than immediately turning to the Gamma function, let's first ask what we would expect these quantities to be, from a combinatorial viewpoint. We use binomial coefficients as a lever, recalling that
$$\binom{n}{k}\ =\ \frac{n!}{k!(n-k)!},$$
which counts the number of ways to choose $k$ objects from a collection of $n$ objects, where order does not matter. For example,
$$\binom{5}{3}\ =\ \frac{5!}{3!2!}\ =\ 10.$$
But what happens if we reverse the roles of $3$ and $5$? How many ways are there to choose five objects from a collection of only three objects? Clearly, there are none. Thus, $\binom{3}{5}\ =\ 0$.
Now suppose we try to use the same factorial formula, even though we have no right to apply it. Let's do so formally and see
$$\binom{3}{5}\ =\ \frac{3!}{5!(-2)!}.$$
The numerator and $5!$ are both positive and finite. If we want this expression to agree with the combinatorial answer of zero, then $(-2)!$ would somehow have to behave like an infinite quantity.

The same reasoning applies more generally. For positive integers $n$ and $k$, we know from counting that
$$\binom{n}{n+k}\ =\ 0.$$
Formally applying the factorial formula gives
$$\binom{n}{n+k}
\ =\ \frac{n!}{(n+k)!(-k)!}.$$
This again suggests that $(-k)!$ should behave like infinity whenever $k$ is a positive integer.

Of course, this is not a proof that negative factorials equal infinity. The usual factorial formula for binomial coefficients was not originally defined when the lower argument exceeds the upper, so we cannot simply substitute these values and treat the resulting expression as an ordinary numerical identity. Instead, we should view this as a consistency check: if we hope to extend the factorial function while preserving the familiar behavior of binomial coefficients, then negative integer factorials ought to behave in some sense like infinite quantities.

But what exactly do we mean by ``infinity?" In complex analysis, the extended complex plane has only one point at infinity, unlike the extended real line, which distinguishes between $+\infty$ and $-\infty$. Since the Gamma function has poles at the non-positive integers, we may write, with this interpretation,
$$(-n)!\ =\ \infty,\qquad n=1,2,3,\ldots.$$
More precisely, these factorials are undefined as finite numbers, but their Gamma-function extensions tend to the single point at infinity in the extended complex plane.

This example illustrates a useful principle in mathematical discovery. When extending a familiar definition, we should not only ask what formulas we can write down, but also whether those formulas agree with properties we already understand. Here, a simple counting problem gives us a clue about what an extension of factorials ought to look like, and the Gamma function passes this consistency check.

\section{The half-derivative and general fractional derivative}
\label{sec:formula}

Our exploration has suggested a particular rule for powers of $x$: due to \eqref{e10} and \eqref{e11}, we guess that
$\Dhalf(x^n)$ should be equal to
$$\frac{\Gamma(n+1)}{\Gamma(n+\frac{1}{2})}x^{n-1/2}.$$
Our guess matches the existing literature. For $x>0$, define the \emph{half-integral} by
\begin{equation}\label{eq:half-integral}
  I^{1/2}f(x) \ :=\ \frac1{\sqrt\pi}\int_0^x\frac{f(t)}{\sqrt{x-t}}\,dt,
\end{equation}
whenever the integral exists. We then define the \emph{half-derivative} by
\begin{equation}\label{eq:RL-half}
  \Dhalf f(x)\ :=\ \frac{d}{dx}I^{1/2}f(x).
\end{equation}
This is the \textit{Riemann-Liouville half-derivative} with starting point $0$. Notice the order: first a weighted integral, then an ordinary derivative. The integral is a natural ``halfway step'' because its kernel contains a square root. For our purpose, the important fact is that it gives exactly the formula we discovered. See \cite{feng2021, OldhamSpanier, Samko} for further background on this theory.

\begin{theorem}[Half-derivative of $x^n$]\label{thm:halfderiv}
For every non-negative integer $n$ and $x>0$, the half-derivative defined in \eqref{eq:RL-half} satisfies
\begin{equation}\label{eq:halfderiv}
  \Dhalf(x^n)\ =\ \frac{\Gamma(n+1)}{\Gamma(n+\tfrac12)}x^{n-1/2}
  \ =\ \frac{n!}{\Gamma(n+\tfrac12)}x^{n-1/2}.
\end{equation}
\end{theorem}

\begin{proof}
Put $t=xu$ in \eqref{eq:half-integral}. We get
\[
 I^{1/2}(x^n)\ =\ \frac{x^{n+1/2}}{\sqrt\pi}
 \underbrace{\int_0^1\frac{u^n}{\sqrt{1-u}}\,du}_{J_n}.
\]
Here $J_0=2$. For $n\geq1$, integration by parts, with $a = u^n$ and $db = du/\sqrt{1-u}$, gives
\[
 J_n\ =\ 2n\int_0^1u^{n-1}\sqrt{1-u}\,du
     \ =\ 2n(J_{n-1}-J_n),
\]
so 
$$J_n\ =\ \frac{n}{n+1/2}J_{n-1}.$$ The Gamma recurrence \eqref{Eq: rec} and $\Gamma(3/2)=\sqrt\pi/2$ now give
\begin{align*}
 J_n&\ =\ \frac{n}{n+1/2} \cdot \frac{n-1}{n-1/2} \cdots \ \frac{2}{3/2}J_{0}\\
 &\ =\ \frac{2\Gamma(n+1)}{\frac{\Gamma(n+3/2)}{\Gamma(3/2)}}\ =\ \frac{\sqrt{\pi}\Gamma(n+1)}{\Gamma(n+3/2)}.
\end{align*}
Thus
\[
 I^{1/2}(x^n)\ =\ \frac{\Gamma(n+1)}{\Gamma(n+3/2)}x^{n+1/2}.
\]
Differentiating and using $\Gamma(n+3/2)=(n+1/2)\Gamma(n+1/2)$ proves the claimed formula.
\end{proof}

\begin{example}
For $n = 1$,
$$\Dhalf(x)\ =\ \frac{1!}{\Gamma(3/2)}\,x^{1/2}\ =\ \frac{1}{\sqrt{\pi}/2}\,x^{1/2}\ =\ \frac{2}{\sqrt{\pi}}\,x^{1/2}.$$
Applying $\Dhalf$ again,
$$\Dhalf\!\left(\frac{2}{\sqrt{\pi}} x^{1/2}\right) \ =\ \frac{2}{\sqrt{\pi}} \cdot \frac{\Gamma(3/2)}{\Gamma(1)}\,x^0 \ =\  \frac{2}{\sqrt{\pi}} \cdot \frac{\sqrt{\pi}/2}{1} \ =\ 1.$$
And indeed, $\frac{d}{dx}(x) = 1$. 
\end{example}

\begin{example}
For $n = 2$, 
\begin{equation}\label{e3}\Dhalf(x^2) \ =\  \frac{2!}{\Gamma(5/2)}\,x^{3/2} \ =\ \frac{2}{3\sqrt{\pi}/4}\,x^{3/2} \ =\ \frac{8}{3\sqrt{\pi}}\,x^{3/2}.\end{equation}
Applying $\Dhalf$ again,
$$\Dhalf\!\left(\frac{8}{3\sqrt{\pi}}x^{3/2}\right)\ =\  \frac{8}{3\sqrt{\pi}}\cdot\frac{\Gamma(5/2)}{\Gamma(2)}\,x \ =\ \frac{8}{3\sqrt{\pi}} \cdot \frac{3\sqrt{\pi}/4}{1}\,x \ =\ 2x,$$
and of course, $\frac{d}{dx}(x^2) = 2x$. 
\end{example}

\subsection{Verification via a second derivation}
\label{sec:second-deriv}

There is a slicker (though less discovery-oriented) way to guess~\eqref{eq:halfderiv}. As 
$$n(n-1)\cdots (n-(k-1))\ =\ \frac{n!}{(n-k)!},$$
we have
\begin{equation}\label{eq:kth-deriv}
  \frac{d^k}{dx^k}(x^n) \ =\  n(n-1)\cdots(n-(k-1))\,x^{n-k}\ =\ \frac{n!}{(n-k)!}\,x^{n-k}.
\end{equation}
Then
\[
  \frac{n!}{(n-k)!} \ =\ \frac{\Gamma(n+1)}{\Gamma(n-k+1)}.
\]
suggests replacing the integer $k$ by a noninteger order and keeping the same quotient:
\[
  \mathscr{D}^k(x^n) \ :=\ \frac{\Gamma(n+1)}{\Gamma(n-k+1)}\,x^{n-k}.
\]
Setting $k = 1/2$ recovers~\eqref{eq:halfderiv}.

This is a useful shortcut for discovering the formula, but it doesn't motivate the appearance of the Gamma function. The integral calculation above supplies that verification when $k=1/2$. When $k>n$ is an integer, the usual derivative is zero; the convention $1/\Gamma(0)=1/\Gamma(-1)=\cdots=0$ also recovers this fact.

\subsection{General fractional derivatives}

The formula for the half-derivative suggests a rule for other orders, and even for powers with exponents that are not integers. In the Riemann-Liouville approach the corresponding power rule is
\begin{equation}\label{eq:fractional-general}
  \displaystyle
  \mathscr{D}^{\alpha}(x^{\beta})
  \ =\ \frac{\Gamma(\beta+1)}{\Gamma(\beta-\alpha+1)}x^{\beta-\alpha},
\end{equation}
for $x>0$, $\beta>-1$, and positive orders $\alpha$, with the reciprocal-Gamma convention from Section~\ref{sec:gamma}. In this paper, we take \eqref{eq:fractional-general} as the natural extension of the power rule; the detailed construction for general $\alpha$ can be found in \cite{OldhamSpanier,Samko}.

For example, three applications of $\mathscr{D}^{1/3}$ to $x^n$ give the ordinary derivative of $x^n$ for nonnegative integers $n$: successive Gamma factors cancel. More broadly, such cancellations explain why fractional orders can often be added when operators are composed. There is one caution: compositions do not work for every function without restriction, and functions with singularities can break our methodology.

\begin{remark}[Why domains matter]\label{rem:composition-caution}
The power $x^{-1/2}$ illustrates the caution. By \eqref{eq:fractional-general},
\[
 \Dhalf(x^{-1/2})
 \ =\ \frac{\Gamma(1/2)}{\Gamma(0)}x^{-1}\ =\ 0.
\]
Taking another half-derivative still gives zero, while
\[
 \frac{d}{dx}x^{-1/2}\ =\ -\frac12x^{-3/2}\ \neq \ 0.
\]
This does not contradict our examples with polynomials: $x^{-1/2}$ is singular at the starting point $x=0$. Thus, applying two half-derivatives to obtain one ordinary derivative is a meaningful goal for well-behaved functions, not an identity to use without checking the setting.
\end{remark}

\begin{remark}[Other orders]
There is no difficulty writing the Gamma quotient in \eqref{eq:fractional-general} for an irrational order such as $\sqrt2$ or $\pi$. Negative orders correspond to fractional integration rather than differentiation; for example, $\mathscr{D}^{-1}f(x)=\int_0^x f(t)\,dt$ in this convention. Different definitions of fractional derivatives are possible, and they need not agree on all functions.
\end{remark}

\subsection{The half-derivative of the constant function}
\label{sec:negfact}

Formula~\eqref{eq:halfderiv} applies to $n = 0$ as well:
\[
  \Dhalf(x^0) \ =\ \Dhalf(1)\ =\ \frac{\Gamma(1)}{\Gamma(1/2)}\,x^{-1/2}\ =\ \frac{1}{\sqrt{\pi}}\,x^{-1/2}.
\]
This is initially surprising: unlike the ordinary derivative, the half-derivative of a constant is not zero. Nevertheless, the composition rule works here. In fact, the integral definition gives an especially simple explanation:
\[
 I^{1/2}(x^{-1/2})\ =\ \frac1{\sqrt\pi}
 \int_0^x\frac{dt}{\sqrt{t}\sqrt{x-t}}
 \ =\ \sqrt\pi.
\]
The last equality follows by the substitution $t=x\sin^2\theta$. Since the half-integral is constant, its derivative is zero. Consequently,
\[
 \Dhalf\bigl(\Dhalf(1)\bigr)
 \ =\ \frac1{\sqrt\pi}\Dhalf(x^{-1/2})\ =\ 0
 \ =\ \frac{d}{dx}(1).
\]
Thus, the half-derivative of a constant need not be constant, even though applying it twice recovers the usual answer.

\subsection{The fractional Leibniz rule}

A natural follow-up question: does the product rule hold for $\Dhalf$? We can test this on monomials.  Take $f(x) = x^a$ and $g(x) = x^b$, with $a,b$ nonnegative integers.  Then $f(x)g(x) = x^{a+b}$, and
\[
  \Dhalf(x^{a+b}) \ =\  \frac{\Gamma(a+b+1)}{\Gamma(a+b+1/2)}\,x^{a+b-1/2}.
\]
On the other hand,
\begin{align*}
  &(\Dhalf f(x))\,g(x) + f(x)\,(\Dhalf g(x))\\
  &\ =\ \frac{\Gamma(a+1)}{\Gamma(a+1/2)}\,x^{a-1/2} \cdot x^b + x^a \cdot \frac{\Gamma(b+1)}{\Gamma(b+1/2)}\,x^{b-1/2}\\
  &\ =\ \left(\frac{\Gamma(a+1)}{\Gamma(a+1/2)} + \frac{\Gamma(b+1)}{\Gamma(b+1/2)}\right) x^{a+b-1/2}.
\end{align*}

These are equal only if 
\begin{equation}\label{e1}\frac{\Gamma(a+b+1)}{\Gamma(a+b+1/2)} \ =\  \frac{\Gamma(a+1)}{\Gamma(a+1/2)} + \frac{\Gamma(b+1)}{\Gamma(b+1/2)},\end{equation} which is false in general.  For example, taking  $a = b = 1$ turns left side of \eqref{e1} into
$$\frac{\Gamma(3)}{\Gamma(5/2)}\ =\ \frac{2}{3\sqrt{\pi}/4}\ =\ \frac{8}{3\sqrt{\pi}},$$
while the right side is
$$2\frac{\Gamma(2)}{\Gamma(3/2)}\ =\ \frac{2}{\sqrt{\pi}/2}\ =\ \frac{4}{\sqrt{\pi}}.$$

\medskip
In fact, a generalized Leibniz rule, going back to Liouville \cite{Liouville1832}, takes the form of an infinite series when the functions are sufficiently regular and the series is convergent:
\begin{equation}\label{e2}\mathscr{D}^\alpha(fg) \ =\ \sum_{k=0}^\infty \binom{\alpha}{k} (\mathscr{D}^{\alpha-k} f)(\mathscr{D}^k g),\end{equation}
where $\binom{\alpha}{k} = \frac{\Gamma(\alpha+1)}{\Gamma(k+1)\Gamma(\alpha-k+1)}$ is the generalized binomial coefficient. The infinite series is not automatically convergent for arbitrary functions, but it becomes a finite sum in our next example. When $\alpha=1$, the terms with $k\geq2$ vanish, recovering the classical product rule. Taking $f=g=x$, the terms with integer $k\geq2$ vanish because
$\mathscr{D}^k(x)=0$. Thus,
\begin{align*}
  \mathscr{D}^{1/2}(x^2)
   &\ =\ \sum_{k=0}^\infty\binom{1/2}{k}
     (\mathscr{D}^{1/2-k}x)(\mathscr{D}^k x)\\
   &\ =\ x\Dhalf(x)+\frac12\mathscr{D}^{-1/2}(x)\\
   &\ =\ \frac{\Gamma(2)}{\Gamma(3/2)}x^{3/2}
     +\frac12\frac{\Gamma(2)}{\Gamma(5/2)}x^{3/2}\\
   &\ =\ \frac{8}{3\sqrt\pi}x^{3/2},
\end{align*}
which matches \eqref{e3}.

\subsection{The half-derivative of \texorpdfstring{$e^x$}{e^x}}
\label{sec:exponential}

With the formula for $\Dhalf(x^n)$ in hand, we can compute the half-derivative of $e^x$ by applying it to the Taylor series term by term:
\[
  \Dhalf(e^x) \ =\ \Dhalf\!\left(\sum_{n=0}^\infty \frac{x^n}{n!}\right) \ =\ \sum_{n=0}^\infty \frac{1}{n!}\,\Dhalf(x^n).
\]
(Appendix~\ref{app:interchange} explains why the term-by-term calculation is justified in this example.) Substituting~\eqref{eq:halfderiv} gives
\begin{equation}\label{eq:halfderivexp}
  \Dhalf(e^x) \ =\  \sum_{n=0}^\infty \frac{1}{n!} \cdot \frac{n!}{\Gamma(n + 1/2)}\,x^{n-1/2}
  \ =\ x^{-1/2} \sum_{n=0}^\infty \frac{x^n}{\Gamma(n + 1/2)}.
\end{equation}

Using the half-integer Gamma values $\Gamma(n + \tfrac{1}{2}) = \frac{(2n-1)!!}{2^n}\sqrt{\pi}$, we can write each term explicitly. Recall that, for a positive integer $m$,
\[
  m!! \ =\ m(m-2)(m-4)\cdots
\]
terminates at $1$ (if $m$ is odd) or $2$ (if $m$ is even).  Also $0!! = 1$ and $(-1)!! = 1$.
As $2^n \cdot n! = (2n)!!$,  we have
\[
  \frac{n!}{\Gamma(n + 1/2)} \ =\  \frac{n!}{\frac{(2n-1)!!}{2^n}\sqrt{\pi}} \ =\  \frac{2^n \cdot n!}{(2n-1)!!\sqrt{\pi}} \ =\ \frac{(2n)!!}{(2n-1)!!\,\sqrt{\pi}}.
\]
As a result,
\begin{equation}\label{eq:halfderivexp2}
  \Dhalf(e^x) \ =\ \frac{x^{-1/2}}{\sqrt{\pi}} \sum_{n=0}^\infty \frac{(2n)!!}{(2n-1)!!}\,\frac{x^n}{n!}.
\end{equation}

This function is not $e^x$ itself, but its half-derivative has a particularly pleasant alternative expression. Recall the \emph{error function} (see \cite[Eq.~7.2.1]{DLMF}), defined by
\[
  \operatorname{erf}(z)\ :=\ \frac2{\sqrt\pi}\int_0^z e^{-u^2}\,du.
\]
Using the integral definition \eqref{eq:half-integral} and the change of variables $u=\sqrt{x-t}$, we obtain
\begin{align*}
 I^{1/2}(e^x)
 &\ =\ \frac{e^x}{\sqrt\pi}\int_0^x\frac{e^{-(x-t)}}{\sqrt{x-t}}\,dt\\
 &\ =\ \frac{2e^x}{\sqrt\pi}\int_0^{\sqrt{x}}e^{-u^2}\,du
 \ =\ e^x\operatorname{erf}(\sqrt{x}).
\end{align*}
Differentiating gives the closed form
\begin{equation}\label{eq:half-exp-erf}
\displaystyle
 \Dhalf(e^x)\ =\ e^x\operatorname{erf}(\sqrt{x})
                +\frac1{\sqrt{\pi x}}.
\end{equation}
Thus, a question about half-derivatives leads us back to the Gaussian function from Section~\ref{sec:gamma}. Applying the half-derivative once more gives $e^x$, as the composition property suggests for this smooth function.

\begin{remark}[Combinatorial interpretation of the double factorial]
The double factorial $(2n-1)!!$ counts the number of ways to partition $\{1,2,\ldots,2n\}$ into $n$ unordered pairs (a perfect matching on $2n$ objects).  This follows from the following. Label person $1$; she must be paired with one of the remaining $2n-1$ people, giving $2n-1$ choices, and the remaining $2n-2$ people form the next subproblem.  Thus the count satisfies the recursion $a_{2n} = (2n-1)\,a_{2n-2}$ with $a_2 = 1$, giving $a_{2n} = (2n-1)!! $.

This count also equals the even moments of the standard normal: if $X \sim N(0,1)$, then $E[X^{2n}] = (2n-1)!!$.
\end{remark}

\subsection{Half-derivative of $\sin x$}\label{Subsec: sin}

Trigonometric functions provide another natural setting
in which to explore fractional differentiation. 
Recall that the ordinary derivatives of $\sin x$ and
$\cos x$ are closely related. Is there a similar
relationship between their half-derivatives?
We invite the reader to explore this question
in Subsection \ref{questions}.
Here we focus
on the half-derivative of $\sin x$, beginning with its familiar
Taylor series:
\begin{equation}\label{Eq: sineTaylor}
\sin x \ = \ \sum_{n=0}^{\infty} \frac{(-1)^n x^{2n+1}}{(2n+1)!}.
\end{equation}
As with $e^x$, a term-by-term application of $\Dhalf$ to \eqref{Eq: sineTaylor} gives
\begin{equation}\label{Eq: useLinearity}
\Dhalf(\sin x) \ = \ \sum_{n=0}^{\infty} \frac{(-1)^n}{(2n+1)!} \Dhalf(x^{2n+1}).
\end{equation}
Next, we use the half-derivative formula for monomials \eqref{eq:halfderiv}, but for odd powers of the form $x^{2n+1}$:
\begin{equation}\label{Eq: MonomialOdd}
\Dhalf(x^{2n+1}) \ =\ \frac{\Gamma(2n+2)}{\Gamma(2n+3/2)} x^{2n+1/2} \ = \ \frac{(2n+1)!}{\Gamma(2n+3/2)} x^{2n+1/2}.
\end{equation}
Substitute \eqref{Eq: MonomialOdd} into \eqref{Eq: useLinearity}:

\begin{align*}
\Dhalf(\sin x) &\ = \ \sum_{n=0}^{\infty} \frac{(-1)^n}{(2n+1)!} \cdot \frac{(2n+1)!}{\Gamma(2n+3/2)} x^{2n+1/2} \\
&\ = \ \sum_{n=0}^{\infty} \frac{(-1)^n}{\Gamma(2n+3/2)} x^{2n+1/2} \\
&\ = \ x^{1/2} \sum_{n=0}^{\infty} \frac{(-1)^n}{\Gamma(2n+3/2)}x^{2n}
\end{align*}
From Appendix \ref{app:gamma}, we have that $\Gamma(2n+3/2) = \frac{(4n+1)!!}{2^{2n+1}}\sqrt{\pi}$, so

\begin{align*}
\Dhalf(\sin x) &\ =\ x^{1/2} \sum_{n=0}^{\infty} \frac{(-1)^n \cdot 2^{2n+1}}{(4n+1)!! \sqrt{\pi}} x^{2n}\nonumber \\
&\ = \ \frac{2x^{1/2}}{\sqrt{\pi}} \sum_{n=0}^{\infty} \frac{(-1)^n (2x)^{2n}}{(4n+1)!!}.
\end{align*}


\section{Summary and open questions}
\label{sec:summary}

\subsection{What we have learned}

Our starting point was a simple question: Can we construct a linear operator $\Dhalf$ that, when applied twice, recovers the ordinary derivative? If so, to what extent is such an operator uniquely determined? Guided by these questions, we obtained the following observations.

\begin{enumerate}
  \item Assuming that the half-derivative of a power function takes the form
  \[
    \Dhalf(x^n)\ =\ A(n)x^{B(n)},
  \]
  we are naturally led to the choice $B(n)=n-\frac12$, corresponding to reducing the exponent by one-half.

  \item With this choice, the composition property leads
  to the functional equation
  \[
    A(n)A\left(n-\frac12\right)\ =\ n
  \]
  for positive integers $n$.

  \item Seeking a solution of the form
  \[
    A(n)\ =\ \frac{f(n)}{f(n-\frac12)}
  \]
  reduces the functional equation to
  \[
    f(n)\ =\ n f(n-1),
  \]
  the familiar factorial recurrence.

  \item Choosing the Gamma function extends the factorial recurrence to half-integers and supplies a natural coefficient; neither the recurrence nor the composition property forces a unique choice.

  \item The identity
  \[
    \Gamma\left(\frac12\right)\ =\ \sqrt{\pi},
  \]
  obtained from the Gaussian integral, allows us to evaluate the Gamma function at positive half-integers
  and reveals a connection with double factorials.

  \item These considerations lead to the classical formula
  \[
    \Dhalf(x^n)
    \ =\
    \frac{\Gamma(n+1)}{\Gamma(n+\frac12)}
    x^{n-\frac12}
  \]
  for non-negative integers $n$, and an elementary integral formula verifies this choice as the Riemann-Liouville half-derivative. We applied power series to derive formulas for the half-derivative of other common functions.
\end{enumerate}

Along the way, we encountered connections among fractional derivatives, the Gamma function, double factorials, combinatorics, the Gaussian distribution, and the error function. Our exploratory axioms do not uniquely define fractional differentiation. Rather, they motivate a natural construction, and an
elementary integral calculation verifies that this construction agrees with the classical Riemann--Liouville
half-derivative.

\subsection{Questions for further exploration}\label{questions}

The ideas developed in this paper suggest several questions that readers may wish to investigate further, and these can serve as starting points for investigative projects in upper-division undergraduate courses.

\begin{enumerate}

  \item \textit{Fractional differentiation of power series.}
  Suppose
  \[
    f(x)\ =\ \sum_{n=0}^{\infty}a_nx^n.
  \]
  Can we define $\Dhalf f$ by applying the half-derivative formula to each monomial? Under what convergence conditions does the resulting series define a function? For which functions does the composition property $\Dhalf(\Dhalf f) =f'$ hold?

  \item \textit{Half-derivatives of trigonometric functions.} Using Euler's identity
  \[
    e^{ix}\ =\ \cos x+i\sin x
  \]
  and the power-series approach to fractional differentiation, investigate $\Dhalf(\sin x)$ and $\Dhalf(\cos x)$. Is there a relationship analogous to the familiar identities
  \[
    (\sin x)'\ =\ \cos x,
    \qquad
    (\cos x)'\ =\ -\sin x?
  \]
  Subsection \ref{Subsec: sin} gives a starting point for a possible way to use series expansions to derive such formulas.

  \item \textit{The fractional product rule.} Verify directly that the ordinary product rule does not extend to half-derivatives by considering $\Dhalf(x\cdot x^2)$. How does the generalized fractional Leibniz rule account for the discrepancy? Under what conditions does its infinite series converge? Does the familiar product rule
\[
\Dhalf(af) \ = \ \Dhalf(a)f+a\Dhalf(f)
\]
hold when $a$ is a constant? If so, how can this
be reconciled with linearity,
$\Dhalf(af)=a\Dhalf(f)$?
What does the answer tell us about $\Dhalf(a)$?

  \item \textit{Derivatives of irrational orders.} Using Formula~\eqref{eq:fractional-general}, extend the power-function rule to suitable real exponents and investigate whether
  \[
    \mathscr{D}^{\sqrt{2}}
    \left(\mathscr{D}^{2-\sqrt{2}}(x^n)\right)
    \ =\ 
    \mathscr{D}^{2}(x^n)
    \ =\ 
    n(n-1)x^{n-2}
  \]
  for integers $n\geq 2$. What restrictions are necessary when composing fractional derivatives of arbitrary orders?

  \item \textit{Non-uniqueness of fractional derivatives.} Our construction was motivated by linearity and the composition property, but these conditions do not uniquely determine $\Dhalf$.  Can one construct another operator satisfying the same requirements? What additional properties might distinguish
  a canonical choice?

  \item \textit{Alternative definitions.}
  The Riemann-Liouville and Caputo fractional derivatives provide two standard approaches to fractional differentiation (see \cite{feng2021} for a summary of these notions  of fractional derivative).
  How do their formulas compare with our construction on monomials? In particular, how do they treat constant functions, and why do they differ?

    \item \textit{Combinatorics.} The technique of \textit{differentiating identities} is well-established with classical derivatives: one takes a binomial coefficient identity and then computes derivatives of both sides of the identity to generate new identities. One can try using this same strategy with fractional derivatives.

  \item \textit{Other applications.}
  Fractional derivatives arise in models of anomalous diffusion, viscoelasticity, and signal processing. Find an example of a physical phenomenon described by a fractional differential equation. What role does the order of differentiation play in the model? It can help to think about what happens to solutions of the equation as the order of differentiation changes, or even as it approaches integer values!

\end{enumerate}

\newpage

\appendix
\section{The Gamma function: selected values and identities}
\label{app:gamma}

\begin{center}
\begin{tabular}{cc}
\toprule
$s$ & $\Gamma(s)$ \\
\midrule
$1$ & $1$ \\
$2$ & $1$ \\
$3$ & $2$ \\
$4$ & $6$ \\
$n$ & $(n-1)!$ \\
\midrule
$1/2$ & $\sqrt{\pi}$ \\
$3/2$ & $\sqrt{\pi}/2$ \\
$5/2$ & $3\sqrt{\pi}/4$ \\
$7/2$ & $15\sqrt{\pi}/8$ \\
$(2n+1)/2$ & $\dfrac{(2n-1)!!}{2^n}\sqrt{\pi}$ \\
\midrule
$0$ & $\text{simple pole}$ \\
$-1/2$ & $-2\sqrt{\pi}$ \\
$-1$ & $\text{simple pole}$ \\
\bottomrule
\end{tabular}
\end{center}

\bigskip
\noindent Key identities (the reflection formula is included for reference, without proof):
\begin{align*}
  \Gamma(s+1) &\ =\ s\,\Gamma(s) && \text{(functional equation)} \\
  \Gamma(n+1) &\ =\  n! && (n \in \mathbb{Z}_{\geq 0}) \\
  \Gamma\!\left(\tfrac{1}{2}\right) &\ =\  \sqrt{\pi} \\
  \Gamma(s)\,\Gamma(1-s) &\ =\  \frac{\pi}{\sin(\pi s)} && \text{(Euler, when }s\notin\mathbb Z\text{)} \\
  \Gamma\!\left(\tfrac{2n+1}{2}\right) &\ =\ \frac{(2n-1)!!}{2^n}\sqrt{\pi}.
\end{align*}

\section{Technical note on interchanging \texorpdfstring{$\Dhalf$}{D\^{}1/2} and infinite sums}
\label{app:interchange}

Throughout Subsection~\ref{sec:exponential}, we applied the half-derivative term by term to the Taylor series of $e^x$. Convergence of the resulting series alone would not justify doing so. Here is a direct argument using only elementary properties of power series.

Fix $R>0$. On $0\leq t\leq R$, the Taylor series
\[
 e^t\ =\ \sum_{n=0}^\infty\frac{t^n}{n!}
\]
converges uniformly. The kernel $(x-t)^{-1/2}$ is integrable for $0\leq t\leq x$, so we may insert the series into the half-integral \eqref{eq:half-integral} and integrate term by term. The calculation in Theorem~\ref{thm:halfderiv} yields
\begin{align*}
 I^{1/2}(e^x)
 &\ =\ \sum_{n=0}^\infty\frac1{n!}I^{1/2}(x^n)\\
 &\ =\ \sum_{n=0}^\infty\frac{x^{n+1/2}}{\Gamma(n+3/2)}.
\end{align*}
For $x>0$, this last series can be differentiated term by term. To see why, factor out $\sqrt{x}$: the remaining power series has terms $x^n/\Gamma(n+3/2)$, and the ratio of successive terms is
\[
 \frac{x}{n+3/2}\ \longrightarrow\ 0.
\]
Both that series and its ordinary derivative therefore converge uniformly by the ratio test on bounded intervals. Differentiating $\sqrt{x}$ times the power series is legitimate on any interval bounded away from $0$, and gives
\[
 \Dhalf(e^x)
 \ =\ \sum_{n=0}^\infty\frac{x^{n-1/2}}{\Gamma(n+1/2)},
\]
which is exactly \eqref{eq:halfderivexp}. No advanced function-space theory is needed for this example. General results about fractional derivatives of infinite sums can be found in \cite{OldhamSpanier,Samko}.

\section*{Acknowledgment of AI Assistance}
Claude \cite{Claude4.5, Anthropic} and ChatGPT \cite{OpenAI} assisted in compiling and organizing material from two recorded lectures and slides by Steven J. Miller and in finding references and checking the exposition of the manuscript. The authors subsequently reviewed and revised the AI-assisted draft. The goals and principal mathematical ideas of the paper were conceived by the authors, who take full responsibility for the mathematical correctness and final content of the manuscript.

The original draft was generated by entering the following prompt into \cite{Anthropic}:
\begin{quote}
    Can you convert the lectures here to a paper? Put in as much  
    detail as you can, the target audience are undergraduate real 
    analysis students and their professors. I have taken the 
    transcript from the lectures and posted here.  I'm attaching 
    the slides from the lectures; it is lecture 6 (starting on 
    page 94) and lecture 7 (ending on page 114).  Write in LaTeX,
    save it and the compiled PDF as HalfDerivMillerLecture\_Claude\_ver01.
    \end{quote}
The attachments refer to a PDF of the aforementioned lecture notes originally and organically written by Steven J. Miller.

\end{document}